\documentclass[a4paper, 11pt]{amsart}
\usepackage{amssymb,latexsym}
\usepackage{url}  

\newtheorem{theorem}{Theorem}

\begin{document}

\title[The 5-Engel identity in $B(2,4)$]%
{Proving the 5-Engel identity in the \\ 2-generator group of exponent four}

\author[C. Ramsay]{Colin Ramsay}

\address{School of Electrical Engineering and Computer Science,
The University of Queensland, Queensland 4072, Australia}

\email{uqcramsa@uq.edu.au}

\date{\today}



\begin{abstract}
It is known that the fifth Engel word $E_5$ is trivial in the 2-generator group
  of exponent four $B(2,4)$, and so can be written as a product of fourth 
  powers.
Explicit products of 250 and 28 powers are known, using fourth powers of words 
  up to lengths four and ten respectively.
Using a reduction technique based on the recursive enumerability of the set of 
  trivial words in a finite presentation we were able to rewrite $E_5$ as a 
  product of 26 fourth powers of words up to length five.
\end{abstract}

\maketitle

\section{Introduction}

If $x$ and $y$ are elements of a group $G$, then the commutator 
  $[x,y] = x^{-1}y^{-1}xy$.
The commutator is also known as the first Engel word $E_1$.
The $n$th Engel word $E_n$, $n>1$, is defined as $[E_{n-1},y]$.
$G$ it is said to satisfy the $n$-Engel identity if $E_n$ is trivial in $G$.

Burnside \cite{Bur} introduced what are now called the free Burnside groups
  of exponent $n$ on $d$ generators $B(d,n)$.
The group $B(2,4)$ is finite, with order $2^{12}$, and satisfies the
  5-Engel identity.
So in a free group, and hence any group, $E_5$ can be expressed as a 
  product of fourth powers.
An explicit expression for $E_5$ was first given in \cite{Hav}, and contained
  250 fourth powers.
A shorter expression, with 28 fourth powers, was given in \cite{Kor}.

Given a finite presentation $P = \langle S | R \rangle$ for a group $G$ it 
  is, in general, an unsolvable problem to decide whether or not a given element
  $g \in G$ is trivial.
As $S$ and $R$ are finite, the set of trivial words in $G$ is recursively 
  enumerable.
(This is also true if $S$ is finite and $R$ itself is recursively enumerable, 
  but we consider only finite presentations here.)

In Section~\ref{sec2} we describe a simple method of generating trivial words
  in $P$, and prove that this enumerates the set of all trivial words.
This engenders a reduction technique for generating proofs that a particular
   word in the free group generated by $S$ is trivial in $G$.
Using this, we were able to generate a new shortest proof that $E_5$ is
  trivial in $B(2,4)$.
Our proof uses 26 fourth powers (instead of 28), and utilises words up to length
  five (instead of ten).
We discuss this proof, and the proofs given in \cite{Hav,Kor}, in 
  Section~\ref{sec3}. 

\section{Generating Proofs}\label{sec2}

Let $F$ be the free group with generating set $S$, and let $R$ be a set of words
  in $F$.
Let $R^\prime$ be the normal closure of the subgroup of $F$ generated by all
  words $r\in R$. 
The group $G$ presented by $\langle S | R \rangle$ is the quotient group 
  $F/R^\prime$. 
The elements of $G$ are the cosets of $R^\prime$, with the words in $R^\prime$
  being the words in $F$ which are trivial in $G$.
So a non-empty freely reduced  word $w$ in $F$ is trivial in $G$ if and 
  only if it can be written as a product of conjugates of relators
\begin{equation}\label{eqn1}
  w = \prod^N_{i=1} \, u_i^{-1} \, r_i^{\pm1} \, u_i ,
\end{equation}
  where $r_i \in R$ and the $u_i$ are words in $F$.

The decomposition of $w$ on the right hand side of (\ref{eqn1}) is a proof that 
  $w$ is trivial in $G$, and is called a \emph{proof word}.
It freely reduces to $w$ and if the $r_i$ are excised (being 
  trivial in $G$) the remainder freely reduces to the empty word $\varepsilon$,
  and thus $w$ is trivial in $G$.

The proof word in (\ref{eqn1}) is in the form of individually conjugated
  relators.
There is typically significant free cancellation in the words
  $u_{i-1} u_{i}^{-1}$, $2 \leqslant i \leqslant N$, and we normally
  work with proof words where these words, and $u_1^{-1}$ and $u_N$, are freely 
  reduced.
We call these \emph{reduced proof words}, and note that to enumerate all the 
  trivial words in $G$ it suffices to enumerate only the reduced proof words.

One way to construct reduced proof words is to start with some relator $r \in R$
  and then perform a series of moves, where a move is conjugation by an element
  of $S \cup S^{-1}$ or appending a relator.
This construction can clearly generate a list of reduced proof words --
  all 1-move proofs, all 2-move proofs, etc.\ 
  -- but it is not obvious that it enumerates all trivial words.
  
\begin{theorem}\label{thm1}
Any reduced proof word $U$ can be generated from $\varepsilon$ by a series of
  relator append and conjugation moves.
\end{theorem}

\begin{proof}
Let $v_1$, $v_i$ ($2 \leqslant i \leqslant N$) and $v_{N+1}$ be the freely 
  reduced forms of the words $u_1^{-1}$, $u_{i-1} u_{i}^{-1}$ and $u_N$ in
  (\ref{eqn1}) respectively.
Let $V$ be the product $v_1v_2\cdots v_{N+1}$.
Note that $V$ freely reduces to $\varepsilon$, but we do not freely 
  reduce it at this stage.
Now consider the product $V^{-1}U$, which equals
\begin{equation}\label{eqn2}
  v_{N+1}^{-1} \cdots v_2^{-1}v_1^{-1} \; 
    v_1 \,r_1^{\pm1}\, v_2 \,r_2^{\pm1} \cdots v_N \,r_N^{\pm1}\, v_{N+1} .
\end{equation}
If the subword $v_1^{-1}v_1$ is cancelled, then the remaining word can
  be generated by a series of moves:\ start with $r_1^{\pm1}$, conjugate by 
  $v_2$, append $r_2^{\pm1}$, etc.
Since $v_1^{-1}$ has been cancelled from $V^{-1}$, the remainder
  $v_{N+1}^{-1} \cdots v_2^{-1}$ must freely reduce to $v_1$.
So we have generated the reduced proof word $U$, as required.
\end{proof}

If a freely reduced word $T \in F$, called the \emph{target} word, is trivial 
  in $G$ then there exist reduced proof words proving this.
We can, in principle, find these by enumerating the trivial words and checking 
  them to see if they equal $T$.
(In fact, any conjugate of $T$ or $T^{-1}$ would suffice.)
If we only require one reduced proof word, say $U$, we can focus our
  search by attempting to reduce $T$ to $\varepsilon$.
Observe that the word $T^{-1}U$ must freely reduce to $\varepsilon$, let $V$ be
  as in the proof of Theorem~\ref{thm1}, and
  consider the product $T^{-1} U V^{-1}$, which equals
\begin{equation}\label{eqn3}
  T^{-1} \; 
  v_1 \,r_1^{\pm1}\, v_2 \,r_2^{\pm1} \cdots v_N \,r_N^{\pm1}\, v_{N+1} \;
  v_{N+1}^{-1}v_N^{-1} \cdots v_1^{-1} . 
\end{equation}

If we now cancel the subword $v_{N+1} v_{N+1}^{-1}$ and conjugate, we obtain
  the expression
\begin{equation}\label{eqn4}
  v_N^{-1} \cdots v_1^{-1} \;
  T^{-1} \; 
  v_1 \,r_1^{\pm1}\, v_2 \,r_2^{\pm1} \cdots v_N \,r_N^{\pm1} .
 \end{equation}
This word can be generated by starting with $T^{-1}$ and then making a series
  of conjugation and append moves.
If the word (\ref{eqn4}) freely reduces to $\varepsilon$ then we have generated 
  a proof that $T$ is trivial in $G$.
Since any conjugate of $\varepsilon$ is trivial, the reduced proof word $U$ can
  be recovered by conjugation, noting that $v_N^{-1} \cdots v_1^{-1}$ freely 
  reduces to $v_{N+1}$.

In the sequel we consider the case of proving that $E_5$ can be
   written as a product of fourth powers in a 2-generator group.
So we fix $S=\{a,b\}$, and adopt the case inversion convention 
  that $A=a^{-1}$ and $B=b^{-1}$.
Our relators are fourth powers of freely and cyclically reduced words (called 
  \emph{base words}), and are finite in number.
We assume that the set of relators has been symmetrized; that is, it includes
  all inverses and rotations of the words in $R$.
If the 5th Engel word $[[[[[a,b],b],b],b],b]$ is expanded out and freely reduced, 
  then a word of length 72 is obtained.
This word is of the form $BBBB \cdots bbbb$, so we use as our target word the
  cyclically reduced version, of length 64, and conjugate the resultant proof
  words by $bbbb$ (which is not treated as a relator).

Our base words are drawn from lists of bracelets on the generators $S$.
Bracelets are equivalence classes of strings under rotations and reversals.
As we are working within a free group (our alphabet is $\{a,A,b,B\}$), we use 
  the inverse of a word instead of its reversal and require our bracelets to be 
  freely and cyclically reduced.
We call the resulting bracelets \emph{reduced bracelets}, and modifying an  
  extant bracelet enumeration algorithm to generate these is described in 
  \cite{Ram}.
The maximum length of the base words in the proof in \cite{Kor} is ten, and in
  Table~\ref{tab2} we list the counts of reduced bracelets on two generators up
  to length ten.
We also include the counts of the Lyndon words, which are the reduced bracelets 
  which are not proper powers.

\begin{table}[t]
\caption{Base word counts (2 generators)}\label{tab2}
\begin{tabular*}{0.8\textwidth}{@{\extracolsep{\fill}}lrrrrrrrrrr}
\ word length & 1 & 2 & 3 & 4 & 5 & 6 & 7 & 8 & 9 & 10\\
\hline
\ reduced bracelets & 2 & 4 & 6 & 13 & 26 & 66 & 158 & 418 & 1098 & 2968 \\
\ Lyndon words      & 2 & 2 & 4 &  9 & 24 & 58 & 156 & 405 & 1092 & 2940 \\
\end{tabular*}
\end{table}
 
\section{Results}\label{sec3}

Havas' proof in \cite{Hav} was derived from the workings of coset enumerations 
  using a presentation for $B(2,4)$ with ten relators, using base words with 
  lengths from one to four.
The presentation was a nine relator presentation for $B(2,4)$ plus an
  extra relator to improve the performance of the coset enumerations.
A range of proofs was obtained by varying the settings of the enumerations, with
  the shortest proof word found having 250 fourth powers.
Note that coset enumeration only uses the relators given in the presentation, 
  so powers not in the presentation do not appear in the proofs produced.

Korlyukov's proof in \cite{Kor} was obtained using a combination of computer and
  hand calculations.
He made use of the recursive nature of the definition of $E_n$ to 
  rewrite $E_5$ in a shortened form using the subword $a^{-1}b^{-1}aba^{-1}ba$.
It is copies of this word or its inverse, combined with $b^n$ for some 
  $n \in \{\pm1, \pm2, \pm3\}$, which yield the 
  base words of lengths eight, nine and ten in his proof word.
The proof proceeds by the judicious introduction of freely trivial words which
  yield fourth powers in the rewritten word.
This was continued until $E_5$ had been rewritten as an element of the 
  normal closure, yielding a proof word  with 28 fourth powers.
Observe that there is no initial presentation.
However an ex post presentation is implicit in the proof word, and is 
  simply the set of distinct relators therein.

To facilitate proof verification and the comparison of proofs,  we put all 
  proof words into a standardised reduced form.
The relators are written out in full and delimited by parentheses, with the 
  strings of conjugating elements outside the relators freely reduced.
The parentheses are not part of the proof, but allow us to distinguish between
   the conjugation and the relators.
If possible, matching generator plus inverse pairs which border a relator are 
  folded into the relator, reducing the length by two symbols and rotating the
  relator (for example, $a(babababa)A$ can be replaced by $(abababab)$).

\begin{table}[t]
\caption{Reduced proof word statistics}\label{tab1}
\begin{tabular*}{0.8\textwidth}{@{\extracolsep{\fill}}lrrrrr}
\ & GH & AVK & CR & \phantom{xxxx}1to4 & inf\\
\hline
\ overall length         & 3180 & 616  & 444  & 716  & 842 \\
\ count of relators      & 250  & 28   & 26   & 48   & 60 \\
\ sum of relator lengths & 1912 & 408  & 272  & 440  & 552 \\
\ mean base word length  & 1.91 & 3.64 & 2.62 & 2.29 & 2.30 \\
\ conjugating pairs      & 384  & 76   & 60   & 90   & 85 \\
\ pairs per relator      & 1.54 & 2.71 & 2.31 & 1.88 & 1.42 \\
\ distinct relators      & 10   & 12   & 13   & 11   & 12 \\
\ group order          & $2^{12}$ & $2^{24}$ & $2^{13}$ & $2^{13}$ & $\infty$ \\
\end{tabular*}
\end{table}

The statistics of all the proof words we discuss are given in Table~\ref{tab1}, 
  with the proofs from \cite{Hav} and \cite{Kor} given in the ``GH'' and ``AVK''
  columns respectively.
The overall length includes the parentheses.
Korlyukov's proof is much shorter than Havas' but has longer 
  base words and more conjugation per relator.
All ten of the relators in Havas' presentation appear in the proof word (this 
  need not be true in general), while Korlyukov's proof has twelve distinct 
  relators and these present a finite group of order $2^{24}$ 
  (that is, $|B(2,4)|^2$).

Using our reduction technique with random selections of base words from the 17
  Lyndon words up to length four, we were able to generate a proof word with 48
  powers.
We do not record this proof word here, but give its statistics in the ``1to4'' 
  column of Table~\ref{tab1}.
It is significantly shorter than Havas' original, and its eleven relators 
  present a finite group of order $2 | B(2,4) |$.
Repeating this procedure using the 41 Lyndon words up to length five, we were 
  able to generate the proof word given below.
\begin{center}\ttfamily
  BBBBAbabABa(bAbAbAbA)(aBaBaaBaBaaBaBaaBaBa)Ab(AbAAbAAbAAbA)aBaaBa \
  (bAAbAAbAAbAA)(aaBaaaBaaaBaaaBa)AbAAAb(AAAA)(abABBabABBabABBabABB) \
  (bbbb)BBaBAb(baBAbaBAbaBAbaBA)abABabA(BBBB)b(bbabAbbabAbbabAbbabA) \
  aBAB(BaBABaBABaBABaBA)abAbabA(bbbb)B(BBabABBabABBabABBabA)aBAbbaBA \
  (bbbb)(BBaBBaBBaBBa)A(bbbb)B(AAAA)b(AAAA)aa(aaBaaBaaBaaB)bA \
  (AbAbAbAb)BaB(aBAbaaBAbaaBAbaaBAba)ABabAAB(aaaa)A(AAbAAbAAbAAb) \
  B(aaaa)(ABABABAB)babaaBAbabABabAbaBABaBAbabABaBAbaBABabbbbb
\end{center}

This proof word has thirteen distinct relators, and these present a finite group
  of order $2 | B(2,4) |$.
Its statistics are given in the ``CR'' column of Table~\ref{tab1}.
It contains 26 fourth powers and uses base words of all lengths from one to 
  five.
Although 26 is only a slight improvement on Korlyukov's 28, the proof word is
  substantially shorter, due mainly to the shorter base words and the reduction 
  in the number of conjugating pairs.
Attempts to generate a shorter prof word by using longer Lyndon words (or 
  reduced bracelets) did not succeed.

The large group in Korlyukov's proof begs the question:\ 
  Can we find a set of base words whose fourth powers
  present an infinite group in which $E_5$ is trivial, and generate a  
  proof therefrom?
Augmenting the procedure used to generate the ``CR'' proof with a check to
  accept only infinite presentations, we were able to generate the required
  base words.
The statistics for an example proof word are given in the ``inf'' column of 
  Table~\ref{tab1}.

Our results demonstrate that, to generate short proofs, it is not necessary to 
  use base words up to length ten or to include base words that contain `long'
   subwords of $E_5$.
Nor is it necessary for the group presented to be $B(2,4)$, or even finite.
The most important factors seem to be the selection of the base words and the 
  inclusion of redundant relators 
  (that is, relators which are derivable from the other relators).
For example, in the presentations implicit in the ``GH'', ``AVK'' and ``CR'' 
  proof words, the 10, 12 and 13 distinct relators (resp.) can be reduced to 
  9, 8 and 8 (resp.) without altering the group presented.

\end{document}